\documentclass[10pt,a4paper]{article}
\usepackage[utf8]{inputenc}
\usepackage[english]{babel}
\usepackage[T1]{fontenc}

\usepackage{amsmath}
\usepackage{amsthm}
\usepackage{amsfonts}
\usepackage{amssymb}
\usepackage{scrextend}
\usepackage[colorlinks=true, allcolors=blue]{hyperref}
\usepackage{cleveref}
\usepackage{graphicx}
\usepackage{enumerate}
\usepackage{lmodern}
\usepackage{wrapfig}
\usepackage{mathrsfs}
\usepackage{subcaption}
\usepackage[onelanguage,ruled,vlined,linesnumbered]{algorithm2e}
\usepackage{siunitx}
\usepackage{pdfpages}
\usepackage{float}
\usepackage[colorinlistoftodos]{todonotes}
\usepackage{comment}
\usepackage[parfill]{parskip}
\usepackage{bm}
\usepackage{xcolor}
\usepackage{makecell}
\usepackage{tablefootnote}

\usepackage{diagbox}

\usepackage{tikz}
\usetikzlibrary{shapes,graphs,graphs.standard}

\newcommand{\ru}[1]{\textbf{R#1}}

\newtheorem{theorem}{Theorem}

\newtheorem{corollary}[theorem]{Corollary}

\newtheorem{proposition}[theorem]{Proposition}
\crefname{proposition}{Proposition}{Propositions}
\newtheorem{lemma}[theorem]{Lemma}
\newtheorem{conjecture}[theorem]{Conjecture}

\crefname{claim}{claim}{claims}
\Crefname{claim}{Claim}{Claims}

\usepackage[left=2cm,right=2cm,top=2cm,bottom=2cm]{geometry}

\DeclareMathOperator{\ad}{ad}
\DeclareMathOperator{\mad}{mad}

\usepackage{authblk}

\title{$2$-distance list $(\Delta+3)$-coloring of sparse graphs}

\author{Hoang La}
\affil{LIRMM, Université de Montpellier, CNRS, Montpellier, France\\ xuan-hoang.la@lirmm.fr}
\begin{document}
  \maketitle

\begin{abstract}
A $2$-distance list $k$-coloring of a graph is a proper coloring of the vertices where each vertex has a list of at least $k$ available colors and vertices at distance at most 2 cannot share the same color. We prove the existence of a $2$-distance list ($\Delta+3$)-coloring for graphs with maximum average degree less than $\frac83$ and maximum degree $\Delta\geq 4$ as well as graphs with maximum average degree less than $\frac{14}5$ and maximum degree $\Delta\geq 6$.
\end{abstract}

\section{Introduction}

A \emph{$k$-coloring} of the vertices of a graph $G=(V,E)$ is a map $\phi:V \rightarrow\{1,2,\dots,k\}$. A $k$-coloring $\phi$ is a \emph{proper coloring}, if and only if, for all edge $xy\in E,\phi(x)\neq\phi(y)$. In other words, no two adjacent vertices share the same color. The \emph{chromatic number} of $G$, denoted by $\chi(G)$, is the smallest integer $k$ such that $G$ has a proper $k$-coloring.  A generalization of $k$-coloring is $k$-list-coloring.
A graph $G$ is {\em $L$-list colorable} if for a
given list assignment $L=\{L(v): v\in V(G)\}$ there is a proper
coloring $\phi$ of $G$ such that for all $v \in V(G), \phi(v)\in
L(v)$. If $G$ is $L$-list colorable for every list assignment $L$ with $|L(v)|\ge k$ for all $v\in V(G)$, then $G$ is said to be {\em $k$-choosable} or \emph{$k$-list-colorable}. The \emph{list chromatic number} of a graph $G$ is the smallest integer $k$ such that $G$ is $k$-choosable.

In 1969, Kramer and Kramer introduced the notion of 2-distance coloring \cite{kramer2,kramer1}. This notion generalizes the ``proper'' constraint (that does not allow two adjacent vertices to have the same color) in the following way: a \emph{$2$-distance $k$-coloring} is such that no pair of vertices at distance at most 2 have the same color (similarly to proper $k$-list-coloring, one can also define \emph{$2$-distance $k$-list-coloring}). The \emph{$2$-distance chromatic number} of $G$, denoted by $\chi^2(G)$, is the smallest integer $k$ so that $G$ has a 2-distance $k$-coloring. We denote $\chi^2_l(G)$ the $2$-distance list chromatic number of $G$.

For all $v\in V$, we denote $d_G(v)$ the degree of $v$ in $G$ and by $\Delta(G) = \max_{v\in V}d_G(v)$ the maximum degree of a graph $G$. For brevity, when it is clear from the context, we will use $\Delta$ (resp. $d(v)$) instead of $\Delta(G)$ (resp. $d_G(v)$). 
One can observe that, for any graph $G$, $\Delta+1\leq\chi^2(G)\leq \Delta^2+1$. The lower bound is trivial since, in a 2-distance coloring, every neighbor of a vertex $v$ with degree $\Delta$, and $v$ itself must have a different color. As for the upper bound, a greedy algorithm shows that $\chi^2(G)\leq \Delta^2+1$. Moreover, this bound is tight for some graphs (see \cite{lmpv19} for some examples).

By nature, $2$-distance colorings and the $2$-distance chromatic number of a graph depend a lot on the number of vertices in the neighborhood of every vertex. More precisely, the ``sparser'' a graph is, the lower its $2$-distance chromatic number will be. One way to quantify the sparsity of a graph is through its maximum average degree. The \emph{average degree} $\ad$ of a graph $G=(V,E)$ is defined by $\ad(G)=\frac{2|E|}{|V|}$. The \emph{maximum average degree} $\mad(G)$ is the maximum, over all subgraphs $H$ of $G$, of $\ad(H)$. Another way to measure the sparsity is through the girth, i.e. the length of a shortest cycle. We denote $g(G)$ the girth of $G$. Intuitively, the higher the girth of a graph is, the sparser it gets. These two measures can actually be linked directly in the case of planar graphs.

\begin{proposition}[Folklore]\label{maximum average degree and girth proposition}
For every planar graph $G$, $(\mad(G)-2)(g(G)-2)<4$.
\end{proposition}

A graph is \emph{planar} if one can draw its vertices with points on the plane, and edges with curves intersecting only at its endpoints. When $G$ is a planar graph, Wegner conjectured in 1977 that  $\chi^2(G)$ becomes linear in $\Delta(G)$:

\begin{conjecture}[Wegner \cite{wegner}]
\label{conj:Wegner}
Let $G$ be a planar graph with maximum degree $\Delta$. Then,
$$
\chi^2(G) \leq \left\{
    \begin{array}{ll}
        7, & \mbox{if } \Delta\leq 3, \\
        \Delta + 5, & \mbox{if } 4\leq \Delta\leq 7,\\
        \left\lfloor\frac{3\Delta}{2}\right\rfloor + 1, & \mbox{if } \Delta\geq 8.
    \end{array}
\right.
$$
\end{conjecture}

The conjecture was proven for some cases (\cite{wegner}, \cite{tho18}, \cite{har16}, \cite{havet}) and some subfamilies of planar graphs \cite{lwz03}.

Wegner's conjecture motivated extensive researches on $2$-distance chromatic number of sparse graphs, either of planar graphs with high girth or of graphs with upper bounded maximum average degree which are directly linked due to \Cref{maximum average degree and girth proposition}. For a survey of the work done on $2$-distance coloring of planar graphs with high girth, see \cite{lmpv19}.

In this article, we prove the following theorems:

\begin{theorem} \label{main theorem1}
If $G$ is a graph with $\mad(G)<\frac83$ and $\Delta(G)\geq 4$, then $\chi^2_l(G)\leq \Delta(G)+3$.
\end{theorem}

\begin{theorem} \label{main theorem2}
If $G$ is a graph with $\mad(G)<\frac{14}5$ and $\Delta(G)\geq 6$, then $\chi^2_l(G)\leq \Delta(G)+3$.
\end{theorem}

\Cref{main theorem1} improves upon a previous result for graphs with $\Delta(G)=5$ \cite{bu15}.

Due to \Cref{maximum average degree and girth proposition}, we get the following corollaries for planar graphs:

\begin{corollary}
If $G$ is a graph with $g(G)\geq 8$ and $\Delta(G)\geq 4$, then $\chi^2_l(G)\leq \Delta(G)+3$.
\end{corollary}

\begin{corollary}
If $G$ is a graph with $g(G)\geq 7$ and $\Delta(G)\geq 6$, then $\chi^2_l(G)\leq \Delta(G)+3$.
\end{corollary}

\paragraph{Notations and drawing conventions.} For $v\in V(G)$, the \emph{2-distance neighborhood} of $v$, denoted $N^*_G(v)$, is the set of 2-distance neighbors of $v$, which are vertices at distance at most two from $v$, not including $v$. We also denote $d^*_G(v)=|N^*_G(v)|$. We will drop the subscript and the argument when it is clear from the context. Also for conciseness, from now on, when we say ``to color'' a vertex, it means to color such vertex differently from all of its colored neighbors at distance at most two. Similarly, any considered coloring will be a 2-distance coloring. We will also say that a vertex $u$ ``sees'' another vertex $v$ if $u$ and $v$ are at distance at most 2 from each other. 

Some more notations:
\begin{itemize}
\item A \emph{$d$-vertex} ($d^+$-vertex, $d^-$-vertex) is a vertex of degree $d$ (at least $d$, at most $d$).
\item A \emph{$k$-path} ($k^+$-path, $k^-$-path) is a path of length $k+1$ (at least $k+1$, at most $k+1$) where the $k$ internal vertices are 2-vertices.
\item A \emph{$(k_1,k_2,\dots,k_d)$-vertex} is a $d$-vertex incident to $d$ different paths, where the $i^{\rm th}$ path is a $k_i$-path for all $1\leq i\leq d$.
\end{itemize}

In both proofs, we will consider $G_1$ (resp. $G_2$) a counter-example to \Cref{main theorem1} (resp. \Cref{main theorem2}) with the fewest number of vertices. The purpose of the proofs is to prove that $G_1$ (resp. $G_2$) cannot exist. We will always start by studying their structural properties, then apply a discharging procedure.

\section{Proof of \Cref{main theorem1}}
\label{sec2}

Since $\mad(G_1)<\frac83$, we have 
\begin{equation}\label{equation1}
\sum_{u\in V(G)}(3d(u)-8) < 0
\end{equation}

We assign to each vertex $u$ the charge $\mu(u)=3d(u)-8$. To prove the non-existence of $G_1$, we will redistribute the charges preserving their sum and obtaining a non-negative total charge, which will contradict \Cref{equation1}. Let us study the structural properties of $G_1$ first.

\subsection{Structural properties of $G_1$\label{tutu1}}

\begin{lemma}\label{connected1}
Graph $G_1$ is connected.
\end{lemma}
Otherwise a component of $G_1$ would be a smaller counterexample.

\begin{lemma}\label{minimumDegree1}
The minimum degree of $G_1$ is at least 2.
\end{lemma}

\begin{proof}
By \Cref{connected1}, the minimum degree is at least 1 or $G_1$ would be a single isolated vertex which is $(\Delta(G_1)+3)$-colorable. If $G_1$ contains a degree 1 vertex $v$, then we can simply remove such vertex and 2-distance color the resulting graph, which is possible by minimality of $G_1$. Then, we add $v$ back and color $v$ (at most $\Delta(G_1)$ constraints and $\Delta(G_1)+3$ colors).
\end{proof}

\begin{lemma}\label{2-path1}
Graph $G_1$ has no $2^+$-paths.
\end{lemma}

\begin{proof}
Suppose that $G_1$ contains a $2^+$-path $uvwx$. We color $G_1-\{v,w\}$ by minimality of $G_1$. Observe that $v$ and $w$ each sees at most $\Delta+1$ colors so they have at least two available colors left each. Thus, we can easily extend the coloring to $v$ and $w$.
\end{proof}

\begin{lemma}\label{(111)1}
Graph $G_1$ has no $(1,1,1)$-vertex.
\end{lemma}

\begin{proof}
Suppose by contradiction that there exists a $(1,1,1)$-vertex $u$ with three $2$-neighbors $u_1$, $u_2$, and $u_3$. We color $G_1-\{u,u_1,u_2,u_3\}$ by minimality of $G_1$, then we extend this coloring to the remaining vertices by coloring $u_1$, $u_2$, $u_3$, and $u$ in this order. Observe that this possible since we have $\Delta+3\geq 7$ colors as $\Delta\geq 4$.
\end{proof}

\begin{lemma}\label{110-001}
Graph $G_1$ has no $3$-vertex with a $2$-neighbor and a $(1,1,0)$-neighbor.
\end{lemma}

\begin{proof}
Suppose by contradiction that there exists a 3-vertex $u$ with a $2$-neighbor $v$ and a $(1,1,0)$-neighbor $w$. Let the $2$-neighbors of $w$ be $w_1$ and $w_2$.

First, observe that if two adjacent $3$-vertices share a common $2$-neighbor, for example, if $u$ is also adjacent to $w_1$, then we color $G_1-\{u,w,w_1\}$ by minimality of $G_1$ and finish by coloring $u$, $w$, and $w_1$ in this order. This is possible since we have $\Delta+3$ colors and $\Delta\geq 4$. Hence, all named vertices are distinct. 

Now, we color $G_1-\{u,v,w,w_1,w_2\}$ by minimality. Let $L(x)$ be the list of available colors left for a vertex $x\in\{u,v,w,w_1,w_2\}$. Since we have $\Delta+3$ colors and $\Delta\geq 4$, $|L(v)|\geq 2$, $|L(u)|\geq 2$, $|L(w)|\geq 4$, $|L(w_1)|\geq 3$, and $|L(w_2)|\geq 3$. We remove the extra colors so that $|L(x)|$ reaches the lower bound for each $x\in\{u,v,w,w_1,w_2\}$. Consider the two following cases.
\begin{itemize}
\item If $L(u)\neq L(v)$, then we color $u$ with $c\in L(u)\setminus L(v)$. We finish by coloring $w_1$, $w_2$, $w$, and $v$ in this order.
\item If $L(u)=L(v)$, we color $w_1$ with $c\in L(w_1)\setminus L(u)$ (which is possible since $|L(w_1)|=3$ and $|L(u)|=2$). Then, we color $w$ with $d\in L(w)\setminus (L(u)\cup\{c\})$ (which is possible as $|L(w)|=4$). Finally, we finish by coloring $w_2$, $u$, and $v$ in this order.
\end{itemize}
We thus obtain a valid coloring of $G_1$, which is a contradiction.
\end{proof}

\begin{lemma}\label{110-00(-3)0-011}
Graph $G_1$ has no $3$-vertex with two $(1,1,0)$-neighbors and another $3$-neighbor.
\end{lemma}

\begin{proof}
Suppose by contradiction that there exists a $3$-vertex $u$ with two $(1,1,0)$-neighbors $v$ and $w$ and another $3$-neighbor $t$. Let $v_1$ and $v_2$ (resp. $w_1$ and $w_2$) be $v$'s (resp. $w$'s) $2$-neighbors.

If $v$ and $w$ share a common $2$-neighbor, say $v_1=w_1$, then we color $G_1-\{u,v,w,v_1,v_2,w_2\}$ by minimality of $G_1$ and finish by coloring $u$, $v$, $w$, $v_2$, $w_2$, and $v_1$ in this order. This is possible since we have $\Delta+3$ colors and $\Delta\geq 4$. Note that this coloring also work when $v_2=w_2$. Hence, all named vertices are distinct.

Now, we color $G_1-\{u,v,w,w_1,w_2\}$ by minimality. Let $L(x)$ be the list of available colors left for a vertex $x\in\{u,v,w,w_1,w_2\}$. Since we have $\Delta+3$ colors and $\Delta\geq 4$, $|L(u)|\geq 2$ (as $d(t)=3$), $|L(v)|\geq 2$, $|L(w)|\geq 4$, $|L(w_1)|\geq 3$, and $|L(w_2)|\geq 3$. Note that we obtain the same lower bounds on the lists of colors as in \Cref{110-001}. Thus, the exact same proof holds and we have a valid coloring of $G_1$, which is a contradiction.
\end{proof}

\subsection{Discharging rules \label{tonton1}}

In this section, we will define a discharging procedure that contradicts the structural properties of $G_1$ (see Lemmas \ref{connected1} to \ref{110-00(-3)0-011}) showing that $G_1$ does not exist. We assign to each vertex $u$ the charge $\mu(u)=3d(u)-8$. By \Cref{equation1}, the total sum of the charges is negative. We then apply the following discharging rules.

\medskip

\begin{itemize}
\item[\ru0] Every $3^+$-vertex gives 1 to each of its $2$-neighbors.
\item[\ru1] Every $4^+$-vertex gives 1 to each of its $3$-neighbors.
\item[\ru2] Every $(0,0,0)$-vertex gives 1 to each of its $(1,1,0)$-neighbors.
\end{itemize}

\subsection{Verifying that charges on each vertex are non-negative} \label{verification}

Let $\mu^*$ be the assigned charges after the discharging procedure. In what follows, we will prove that: $$\forall u \in V(G_1), \mu^*(u)\ge 0.$$

Let $u\in V(G_1)$.

\textbf{Case 1:} If $d(u) = 2$, then $u$ receives charge 1 from each endvertex of the 1-path it lies on by \ru0 (as there are no $2^+$-path by \Cref{2-path1}); Thus we get $\mu^*(u) = \mu(u) + 2\cdot 1 = 3\cdot 2 - 8 + 2 = 0$.

\textbf{Case 2:} If $d(u)=3$, then $\mu(u)=3\cdot 3 - 8 = 1$. Since there are no $2^+$-paths due to \Cref{2-path1} and no $(1,1,1)$-vertices due to \Cref{(111)1}, we have the following cases.
\begin{itemize}
\item If $u$ is a $(1,1,0)$-vertex, then $u$ gives 1 to each of its two $2$-neighbors by \ru0. At the same time, $u$ also receives 1 from its $3^+$-neighbor $v$ by \ru1 or \ru2 as $v$ is either a $4^+$-vertex or a $(0,0,0)$-vertex by \Cref{110-001}. To sum up,
$$ \mu^*(u)\geq 1-2\cdot 1 + 1 = 0.$$
\item If $u$ is a $(1,0,0)$-vertex, then $u$ only gives 1 to its $2$-neighbor by \ru0. Hence,
$$ \mu^*(u)\geq 1-1 = 0.$$
\item If $u$ is a $(0,0,0)$-vertex, then $u$ only gives charge to $(1,1,0)$-vertices by \ru2. Let $t$, $v$, and $w$ be $u$'s $3^+$-neighbors. 

If $u$ is adjacent to a $4^+$-neighbor then it receives 1 by \ru1 and at worst, it gives 1 to each of the two other neighbors by \ru2. As a result,
$$ \mu^*(u)\geq 1+1-2\cdot 1 = 0. $$

If $u$ is adjacent to three $3$-vertices, then at most one of them can be a $(1,1,0)$-vertex due to \Cref{110-00(-3)0-011}. So, $u$ only gives at most 1 to a $(1,1,0)$-neighbor by \ru2. Consequently,
$$ \mu^*(u)\geq 1-1 = 0. $$
\end{itemize}

\textbf{Case 3:} If $4\leq d(u)\leq \Delta$, then, at worst, $u$ gives 1 to each of its neighbors by \ru0 and \ru1. As a result,
$$ \mu^*(u)\geq 3d(u)-8-d(u) \geq 2\cdot 4 - 8 = 0.$$

To conclude, we started with a charge assignment with a negative total sum, but after the discharging procedure, which preserved this sum, we end up with a non-negative one, which is a contradiction. In other words, there exists no counter-example $G_1$ to \Cref{main theorem1}. 

\section{Proof of \Cref{main theorem2}}
\label{sec3}

Since $\mad(G_2)<\frac{14}5$, we have 
\begin{equation}\label{equation2}
\sum_{u\in V(G)}(5d(u)-14) < 0
\end{equation}

We assign to each vertex $u$ the charge $\mu(u)=5d(u)-14$. To prove the non-existence of $G_2$, we will redistribute the charges preserving their sum and obtaining a non-negative total charge, which will contradict \Cref{equation2}. Let us study the structural properties of $G_2$ first.

\subsection{Structural properties of $G_2$\label{tutu2}}

Observe that in Lemmas \ref{connected1} to \ref{(111)1}, we only use the fact that $G_1$ is a minimal counter-example to \Cref{main theorem1} and that we have $\Delta(G_1)+3$ colors, which is at least 7. Consequently, since $G_2$ is a minimal counter-example to \Cref{main theorem2} and we have $\Delta(G_2)+3\geq 9$ colors as $\Delta(G_2)\geq 6$, the same arguments allow us to obtain Lemmas \ref{connected2} to \ref{(111)2}:

\begin{lemma}\label{connected2}
Graph $G_2$ is connected.
\end{lemma}

\begin{lemma}\label{minimumDegree2}
The minimum degree of $G_2$ is at least 2.
\end{lemma}

\begin{lemma}\label{2-path2}
Graph $G_2$ has no $2^+$-paths.
\end{lemma}

\begin{lemma}\label{(111)2}
Graph $G_2$ has no $(1,1,1)$-vertex.
\end{lemma}

\begin{lemma}\label{011-5}
A $(1,1,0)$-vertex can only share a common $2$-neighbor with a $\Delta$-vertex.
\end{lemma}

\begin{proof}
Suppose by contradiction that there exists a $(1,1,0)$-vertex $u$, with two $2$-neighbors $v$ and $w$, and let $x$ be the other endvertex of the $1$-path $uwx$ with $d(x)\leq \Delta-1$. We color $G_2-\{u,v,w\}$ by minimality of $G_2$. Then, it suffices to finish coloring $u$, $v$, and $w$ in this order as we have $\Delta+3$ colors. 
\end{proof}

\begin{lemma}\label{110-4}
A $(1,1,0)$-vertex has a $(\Delta-1)^+$-neighbor.
\end{lemma}

\begin{proof}
Suppose by contradiction that there exists a $(1,1,0)$-vertex $u$ with two $2$-neighbors $v$ and $w$ and a $(\Delta-2)^-$-neighbor $t$. We color $G_2-\{u,v,w\}$ by minimality of $G_2$. Then, it suffices to finish coloring $v$, $w$, and $u$ in this order as we have $\Delta+3$ colors.
\end{proof}

\begin{lemma}\label{3-010-3}
A $(1,0,0)$-vertex cannot have two $3$-neighbors.
\end{lemma}

\begin{proof}
Suppose by contradiction that there exists a $(1,0,0)$-vertex $u$ with a $2$-neighbor $v$ and two $3$-neighbors. We color $G_2-\{v\}$ by minimality of $G_2$. We uncolor $u$ then it suffices to finish coloring $v$ and $u$ in this order as we have $\Delta+3$ colors and $\Delta\geq 6$.
\end{proof}

\begin{lemma}\label{4-01(-5)0-3}
A $(1,0,0)$-vertex with a $3$-neighbor and a $4$-neighbor can only share a common $2$-neighbor with a $\Delta$-vertex.
\end{lemma}

\begin{proof}
Suppose by contradiction that there exists a $(1,0,0)$-vertex $u$, with a $2$-neighbor $v$, a $3$-neighbor, and a $4$-neighbor, and let $x$ be the other endvertex of the $1$-path $uvx$ with $d(x)\leq \Delta-1$. We color $G_2-\{v\}$ by minimality of $G_2$. We uncolor $u$ then it suffices to finish coloring $u$ and $v$ in this order as we have $\Delta+3$ colors and $\Delta\geq 6$.
\end{proof}

\begin{lemma}\label{1111-5}
A $(1,1,1,1)$-vertex can only share a common $2$-neighbor with a $\Delta$-vertex.
\end{lemma}

\begin{proof}
Suppose by contradiction that there exists a $(1,1,1,1)$-vertex $u$, with four $2$-neighbors $v_1$, $v_2$, $v_3$, and $v_4$, and let $x$ be the other endvertex of the $1$-path $uv_4x$ with $d(x)\leq \Delta-1$. We color $G_2-\{u,v_1,v_2,v_3,v_4\}$ by minimality of $G_2$. Then, it suffices to finish coloring $v_1$, $v_2$, $v_3$, $v_4$, and $u$ in this order as we have $\Delta+3$ colors and $\Delta\geq 6$.
\end{proof}

\begin{lemma}\label{4-1110-3}
A $(1,1,1,0)$-vertex with a $3$-neighbor can only share a common $2$-neighbor with a $(\Delta-1)^+$-vertex.
\end{lemma}

\begin{proof}
Suppose by contradiction that there exists a $(1,1,1,0)$-vertex $u$, with three $2$-neighbors $v_1$, $v_2$, and $v_3$ and a $3$-neighbor, and let $x$ be the other endvertex of the $1$-path $uv_3x$ with $d(x)\leq \Delta-2$. We color $G_2-\{u,v_1,v_2,v_3\}$ by minimality of $G_2$. Then, it suffices to finish coloring $v_1$, $v_2$, $u$, and $v_3$ in this order as we have $\Delta+3$ colors and $\Delta\geq 6$.
\end{proof}

\subsection{Discharging rules \label{tonton1}}

In this section, we will define a discharging procedure that contradicts the structural properties of $G_2$ (see Lemmas \ref{connected2} to \ref{4-1110-3}) showing that $G_2$ does not exist. We assign to each vertex $u$ the charge $\mu(u)=5d(u)-14$. By \Cref{equation2}, the total sum of the charges is negative. We then apply the following discharging rules:

\medskip

\begin{itemize}
\item[\ru0] Every $3^+$-vertex gives 2 to each of its $2$-neighbors.
\item[\ru1] 
\begin{itemize}
\item[(i)] Every $4$-vertex gives $\frac12$ to each of its $3$-neighbors.
\item[(ii)] Every $5^+$-vertex gives 2 to each of its $3$-neighbors.
\end{itemize}
\item[\ru2] Let $uvw$ be a $1$-path:
\begin{itemize}
\item[(i)] If $d(u)=5$ and $d(w)\leq 4$, then $u$ gives $\frac15$ to $w$.
\item[(ii)] If $d(u)\geq 6$ and $d(w)\leq 4$, then $u$ gives $\frac23$ to $w$.
\end{itemize}
\end{itemize}

\subsection{Verifying that charges on each vertex are non-negative} \label{verification}

Let $\mu^*$ be the assigned charges after the discharging procedure. In what follows, we will prove that: $$\forall u \in V(G_2), \mu^*(u)\ge 0.$$

Let $u\in V(G_2)$.

\textbf{Case 1:} If $d(u) = 2$, then $u$ receives charge 2 from each endvertex of the path it lies on by \ru0 (as there are no $2^+$-path by \Cref{2-path2}); Thus we get $\mu^*(u) = \mu(u) + 2\cdot 2 = 5\cdot 2 - 14 + 2\cdot 2 = 0$.

\textbf{Case 2:} If $d(u)=3$, then recall $\mu(u) = 5\cdot 3 - 14 = 1$. Since there are no $2^+$-paths due to \Cref{2-path2} and no $(1,1,1)$-vertices due to \Cref{(111)2}, we have the following cases.
\begin{itemize}
\item If $u$ is a $(1,1,0)$-vertex, then $u$ gives 2 to each of its two $2$-neighbors. At the same time, the other endvertices of $u$'s incident 1-paths must be $\Delta$-vertices due to \Cref{011-5}. So, $u$ receives $\frac23$ from each of these $\Delta$-endvertices by \ru2(ii) as $\Delta\geq 6$. Moreover, $u$'s $3^+$-neighbor must be a $(\Delta-1)^+$-vertex due to \Cref{110-4}. As a result, $u$ also receives $2$ from its $(\Delta-1)^+$-neighbor as $\Delta\geq 6$. To sum up,
$$ \mu^*(u)\geq 1 - 2\cdot 2 + 2\cdot \frac23 + 2 = \frac13. $$

\item If $u$ is a $(1,0,0)$-vertex, then $u$ only gives 2 to its $2$-neighbor by \ru0. We distinguish the two following cases.
\begin{itemize}
\item If $u$ has a $5^+$-neighbor, then $u$ receives 2 by \ru1(ii). Thus,
$$ \mu^*(u)\geq 1 - 2 + 2 = 1.$$
\item If $u$ only has $4^-$-neighbors, then $u$ must have at least one $4$-neighbor due to \Cref{3-010-3}. 

Now, if $u$ has two $4$-neighbors, then it receives $\frac12$ twice by \ru1(i). Hence,
$$ \mu^*(u)\geq 1- 2 + 2\cdot\frac12 = 0.$$

If $u$ has exactly one $4$-neighbor and the other one is a $3$-neighbor, then the other endvertex of the 1-path incident to $u$ must be a $\Delta$-vertex due to \Cref{4-01(-5)0-3}. So, $u$ receives $\frac12$ from its $4$-neighbor by \ru1(i) and $\frac23$ from the $\Delta$-endvertex by \ru2(ii) as $\Delta\geq 6$. To sum up,
$$ \mu^*(u)\geq 1 - 2 + \frac12 + \frac23 = \frac16. $$
\end{itemize}

\item If $u$ is a $(0,0,0)$-vertex, then $u$ does not gives any charge away. Thus, $$ \mu^*(u)=\mu(u)=1.$$
\end{itemize}

\textbf{Case 3:} If $d(u)=4$, then recall $\mu(u)=5\cdot 4 - 14 = 6$ and observe that $u$ only gives charge 2 or $\frac12$ away respectively by \ru0 or \ru1(i). We have the following cases.
\begin{itemize}
\item If $u$ is a $(1,1,1,1)$-vertex, then $u$ gives $2$ to each of its four $2$-neighbors by \ru0. At the same time, the other endvertices of the 1-paths incident to $u$ are all $\Delta$-vertices due to \Cref{1111-5}. As a result, $u$ also receives $\frac23$ from each of the four $\Delta$-endvertices by \ru2(ii). To sum up,
$$ \mu^*(u)\geq 6 - 4\cdot 2 + 4\cdot \frac23 = \frac23.$$ 

\item If $u$ is a $(1,1,1,0)$-vertex, then $u$ gives $2$ to each of its three $2$-neighbors by \ru0. Let $v$ be the $3^+$-neighbor.

If $v$ is a $4^+$-vertex, then $u$ does not give anything to $v$. Thus,
$$ \mu^*(u)\geq 6-3\cdot 2 = 0.$$

If $v$ is a $3$-vertex, then $u$ gives $\frac12$ to $v$ by \ru1(i). Due to \Cref{4-1110-3}, the other endvertices of the 1-paths incident to $v$ must be $(\Delta-1)^+$-vertices. As a result, $v$ receives at least $\frac15$ from each of the three $(\Delta-1)^+$-endvertices as $\Delta\geq 6$ by \ru2. To sum up,
$$ \mu^*(u)\geq 6 - 3\cdot 2 - \frac12 + 3\cdot\frac15 = \frac1{10}.$$

\item If $u$ is a $(1^-,1^-,0,0)$-vertex, then at worst $u$ gives 2 twice by \ru0 and $\frac12$ twice by \ru1(i). Thus,
$$ \mu^*(u)\geq 6 - 2\cdot 2 - 2\cdot\frac12 = 1.$$
\end{itemize}

\textbf{Case 4:} If $d(u)=5$, then, at worst, $u$ gives $2+\frac15$ away along each incident edge by \ru0 (or \ru1(ii)) and \ru2(i). As a result,
$$ \mu^*(u)\geq 5d(u)-14- \left(2+\frac15\right)d(u) = \left(2+\frac45\right)\cdot 5 - 14 = 0.$$

\textbf{Case 5:} If $6\leq d(u)\leq \Delta$, then, at worst, $u$ gives $2+\frac23$ away along each incident edge by \ru0 (or \ru1(ii)) and \ru2(ii). As a result,
$$ \mu^*(u)\geq 5d(u)-14-\left(2+\frac23\right)d(u) \geq \left(2+\frac13\right)\cdot 6 - 14 = 0.$$

To conclude, we started with a charge assignment with a negative total sum, but after the discharging procedure, which preserved this sum, we end up with a non-negative one, which is a contradiction. In other words, there exists no counter-example $G_2$ to \Cref{main theorem2}.

\bibliographystyle{plain}

\end{document}